\documentclass[12pt]{article}
\usepackage{srcltx}
\usepackage{times}
\usepackage{amsmath}
\usepackage{latexsym}
\usepackage{amsfonts}
\usepackage{color}
\usepackage{amssymb}
\usepackage[pdftex]{graphicx}
\usepackage{latexsym}

\newcounter{theorem}

\renewcommand{\thetheorem}{\arabic{section}.\arabic{theorem}}

\newenvironment{thm}[1]{\par\addvspace{0.5cm}
    \begin{sloppypar}\refstepcounter{theorem}%
    {\bf #1 \thetheorem.}\it{}}{\end{sloppypar}}

\newcommand{\eh}{\hfill}\newlength{\sperr}

\newenvironment{theorem}{\begin{thm}{Theorem}} {\end{thm}}
\newenvironment{lemma}{\begin{thm}{Lemma}} {\end{thm}}
\newenvironment{corollary}{\begin{thm}{Corollary}} {\end{thm}}

\newenvironment{defi}[1]{\par\addvspace{0.5cm}
\begin{sloppypar}\refstepcounter{theorem}%
{\bf #1 \thetheorem.}\rm{}}{\end{sloppypar}}

\newenvironment{example}{\begin{defi}{Example}}{\end{defi}}
\newenvironment{remark}{\begin{defi}{Remark}}{\end{defi}}

\newenvironment{observation}{\begin{thm}{Observation}} {\end{thm}}

\newenvironment{proof}{{\settowidth{\sperr}{\rm Proof}
\par\addvspace{0.3cm}\parbox[t]{1.3\sperr}{\rm P\eh r\eh o\eh o\eh f\eh. }%
}}{\nopagebreak\mbox{}\hfill $\Box$\par\addvspace{0.25cm}}

\newcommand{\vs}{\vspace}
\newcommand{\al}{\alpha}
\newcommand{\bt}{\beta}

\newcommand{\Gm}{\Gamma}

\newcommand{\intl}{\int\limits}

\begin{document}

\begin{center}
\textbf{  \large   Sharpness of some Hardy-type inequalities}

\end{center}


\begin{center}
\textbf{ Lars-Erik Persson$^{1,2}$, Natasha Samko$^{1}$ and George Tephnadze$^{3}$}
\end{center}
 \textit{ $^{1}$UiT The Arctic University of Norway, P.O. Box 385, N-8505, Narvik, Norway,\\
 $^{2}$ Karlstad University, 65188 Karlstad, Sweden,\\
$^{3}$ The University of Georgia,  77a Merab Kostava St, Tbilisi, 0128, Georgia.\\ }

\begin{abstract}
Abstract: The current status concerning Hardy-type inequalities with sharp constants is presented and described in a unified convexity way. In particular, it is then natural to replace the Lebesgue measure $dx$ with the Haar measure $dx/x.$ There are also derived some new two-sided Hardy-type inequalities for monotone functions, where not only the two constants are sharp but also where the involved function spaces are (more) optimal. As applications, a number of both well-known and new Hardy-type inequalities are pointed out. And, in turn, these results are used to derive some new sharp information concerning sharpness in the relation between different quasi-norms in Lorentz spaces.
\end{abstract}

\vs{4mm}

\noindent{\bf 2020 Mathematics Subject Classification:}  26D10, 46E30.

\noindent{\bf Key words and phrases:} Inequalities, Hardy-type inequalities, sharp constants, optimal target function, Lorentz spaces.

\section{ Introduction}
 \setcounter{equation}{0}\setcounter{theorem}{0}

\quad \   The continuous Hardy inequality from 1925 (see \cite{H}) informs us if $f$ is non-negative $p$-integrable function on $(0,\infty),$ then $f$ is integrable over the interval $(0,x)$ for each positive $x$ and
\begin{eqnarray}\label{1.1}
\int_0^\infty\left(\frac{1}{x}\int_0^x f(y)dy\right)^pdx\leq \left(\frac{p}{p-1}\right)^p\int_0^\infty f^p(x)dx, \ \ p>1.
\end{eqnarray}

The development of the famous Hardy inequality in both discrete and continuous forms during the period 1906 to 1928 has its own history or, as we call it, prehistory. Contributions of mathematicians other than G.H. Hardy such as E. Landay, G. Polya, E. Schur and M. Reisz, are important here. This prehistory was described in \cite{KMP}.

The first weighted version of \eqref{1.1} was proved by Hardy himself in 1928 (see \cite{H3}):
\begin{equation}\label{1.2}
\int_0^\infty \left(\frac{1}{x} \int_0^x f(y)dy \right)^p x^a dx\le  \left(\frac{p}{p-1-a}\right)^p \int_0^\infty f^p(x) x^a dx
\end{equation}
where $f$ is a measurable and non-negative function on $(0,\infty)$ whenever
$a<p-1, p>1.$

In the remarkable further development to which today is called Hardy-type inequalities, in the case of weighted Lebesgue spaces mostly the Lebesgue measure $dx$ is used (see the books \cite{HaLiPoB}, \cite{KMPB},   \cite{KuMaLEP} and \cite{KPNS} and the references therein). One basic idea in this paper is to use convexity and then it is more natural to instead use the measure $dx/x$ (=Haar measure when the underlying group is $R+$). Moreover, this way to consider the situation helps us to easier investigate and describe the sharpness in Hardy-type inequalities. In this theory of Hardy-type inequalities (between weighted Lebesgue spaces) we usually have good estimates of the sharp constant (= the operator norm or quasi-norm). However, in very few cases the sharp constant is known.

In this paper we describe and/or derive most of these Hardy-type inequalities in the convexity $(dx/x)$ frame described above. Moreover, we concentrate also on the problem to derive the corresponding reversed inequalities in cones of monotone functions. And still with sharp constants. It turns out that our approach also implies that the sharpness can be further improved in special situations e.g. to not only have sharp constant(s) but also by involving more optimal function spaces, sometimes even with optimal so called target functions involved. In order to illustrate this idea we present the following introductory example:

\begin{example}\label{Example1.1} The inequality \eqref{1.2} holds also if the interval $(0, \infty)$ is replaced by $(0,\ell),$ $0 < \ell \leq \infty,$ and still the constant
$$C =\left(\frac{p}{p-1-\alpha}\right)^p$$
is sharp. However also the following “sharper” inequality is known (see \cite{LEPNS12} and c.f. Theorem 2.3 a) in the book \cite{KPNS}):
 \begin{equation}\label{1.3}
\int_0^\ell \left(\frac{1}{x} \int_0^x f(y)dy
\right)^p x^{\alpha}dx\le \left(\frac{p}{p-1-a}\right)^p \int_0^\ell f^p(x)x^{\alpha} \left[1-\left(\frac{x}{\ell}\right)^{\frac{p-1-a}{p}} \right]dx,
\end{equation}
where $\alpha < p-1, p>1,$ and still the constant $C=\left(\frac{p}{p-1-a}\right)^p$  is sharp. Moreover, we note that in the cone of non-increasing functions \eqref{1.2} holds in the reversed direction with the constant $C = 1.$ But indeed  \eqref{1.3} holds also in the reversed direction with the sharp constant $C = p/(p-1-a)>1$  whenever $\alpha > -1$. See our Theorem 3.2 a). In such a situation when both constants are sharp we say that the involved weight function
$$g(x) := 1-\left(\frac{x}{\ell}\right)^{\frac{p-1-a}{p}}< 1, \ x<\ell,$$
is the “optimal target function”.
\end{example}

The paper is organized as follows: In Section 2 we present the mentioned convexity approach to derive power weighted Hardy-type inequalities and some of its consequences. Here, and in the sequel, it turns out that this convexity approach makes it natural to present such inequalities by using the Haar measure $dx/x$ instead of the Lebesgue measure dx. In Section 3 we derive some new sharp reversed Hardy-type inequalities on cones of monotone functions. Section 4 is used to present and discuss some new applications e.g. concerning two-sided Hardy-type inequalities where both constants are sharp and, moreover, the actual inequalities are further sharpened by pointing out (more) optimal involved function spaces. These results, in its turn, make it possible to derive some new results concerning comparisons of different norms in Lorentz spaces. Finally, Section 5 is reserved for some concluding remarks and for presenting and/or deriving some further sharp Hardy-type inequalities.

\section{ A convexity approach to derive sharp power weighted Hardy-type inequalities} \label{S2}
\setcounter{equation}{0}\setcounter{theorem}{0}

\qquad The fact that the concept of convexity can be used to prove several inequalities, both classical and new ones, was of course known by Hardy himself. For example in the famous book \cite{HaLiPoB} this concept and the more or less equivalent Jensen inequality was frequently used. Hence, it may be surprising that Hardy himself never discovered that also his famous inequality in both original (see \eqref{1.1}) and power weighted form (see e.g. \eqref{1.2}) follow more or less directly as described below. Concerning convexity and its applications e.g. to prove inequalities we refer to the recent book \cite{NikPers},  the papers \cite{LEPNS12}, \cite{LEPSS15}, and the references therein.

\subsection{  A new look on the inequalities \eqref{1.1} and \eqref{1.2} \label{ss2.1}}
 \begin{observation}\label{o181}
 We note that for $p>1$
$$
 \int_0^\infty
\left(\frac{1}{x}\int_0^xf(y)dy \right)^p dx \le
\left(\frac{p}{p-1}\right)^p \int_0^\infty f^p(x)dx,
$$
$$\Leftrightarrow $$
\begin{equation}\label{2.1}
 \int_0^\infty \left(\frac{1}{x}\int_0^x g(y)dy
\right)^p \frac{dx}{x} \le 1\cdot \int_0^\infty
g^p(x)\frac{dx}{x},
\end{equation}
where
 $f(x) = g (x^{1-1/p})x^{-1/p}. $
\end{observation}

This means that Hardy's inequality \eqref{1.1} is
equivalent to \eqref{2.1} for $p>1$  and, thus, that  Hardy's
inequality can be proved in the following simple way (see form
\eqref{2.1}): By Jensen's inequality and Fubini's theorem we have
that
\begin{eqnarray}\label{2.2}
\int_0^\infty\left(\frac{1}{x}\int_0^x
g(y)dy\right)^p\frac{dx}{x}&\le&
\int_0^\infty\left(\frac{1}{x}\int_0^x
g^p(y)dy\right)\frac{dx}{x}\\ \notag
&= & \int_0^\infty g^p(y)\int_y^\infty
\frac{dx}{x^2} dy = \int_0^\infty g^p(y)\frac{dy}{y}.
\end{eqnarray}

By instead making the substitution
$$f(t) = g(t^{\frac{p-1-a}{p}}) t^{-\frac{1+a}{p}}$$ in \eqref{1.2} we see that also this inequality is equivalent to \eqref{2.1}. These facts imply especially the following:

\ (a) \ Hardy's inequalities \eqref{1.1} and \eqref{1.2} hold also for $p<0$ (because the function $\varphi(u) = u^p$ is convex also for $p<0$) and hold in the reverse direction for $0<p<1$ (with sharp constants $ \left(\frac{p}{1-p} \right)^p$ and $\left(\frac{p}{a+1-p} \right)^p, a > p-1,$ respectively).

\ (b) \ The inequalities \eqref{1.1} and \eqref{1.2} are equivalent, since both are equivalent  to \eqref{2.1}

\ (c) \ The inequality \eqref{2.1} holds also with equality for $p=1,$ which gives us a possibility to interpolate and get more information about the mapping properties of the Hardy operator. In particular, we can use interpolation theory to see that in fact the Hardy operator $H$ maps each interpolation space $B$ between $L_1 \left((0, \infty), \frac{dx} x \right)$ and $L_\infty \left((0, \infty), \frac{dx} x \right)$ into $B,$ i.e. that the following more general Hardy type inequality holds:
$$ \| Hf\|_B \le C\| f\|_B. $$

 \subsection{ Further consequences of the new look presented in Section \ref{ss2.1}}\label{ss2.2}

\qquad  For the finite interval case we need the following extention of our basic (convexity) form of Hardy´s inequality presented in Section 2.1.

 \begin{theorem} \label{l2.1}
Let $g$ be a non-negative and measurable function on $(0,\ell), 0<\ell \le \infty.$\\
 \vspace{2mm}
   a) \ \ \ \ If $p<0$ or $p\ge 1,$ then

 \begin{equation}\label{2.3}
 \int_0^\ell \left(\frac{1}{x} \int_0^x g(y)dy
\right)^p \frac{dx}{x}\le 1 \cdot \int_0^\ell g^p(x) \left(
1-\frac{x}{\ell} \right)\frac{dx}{x}.
\end{equation}
(In the case $p<0$ we assume that $g(x)>0, 0 < x \le \ell$).\\
\vspace{4mm}
  b) \ \ \ \ If \ $0<p\le 1,$ then \eqref{2.2} holds in the reversed
  direction.\\
  \vspace{4mm}
c) \ \ \ \ The constant $C=1$ is sharp in both \ a) \ and \
  b).
\end{theorem}
\begin{proof}
Proof a) The proof only consists of an obvious modifications of \eqref{2.2}.

b) Since Jensen´s inequality holds in the reversed direction for the comcave function
$$\phi(u) = u^p, \ \ 0 <p \leq 1,$$
the proof follows in the same way.

c) Assume that \eqref{2.3} with the constant $1$ replaced by some constant $C, \ 0 < C < 1.$ By applying \eqref{2.3} with the test functions
$g(x) = x^\alpha, \  0 \leq x\leq \ell, \  a>0,$
a simple calculation shows that
$$\frac{\alpha p+1}{\left(\alpha+1\right)^p}\leq C < 1$$
so by choosing $\alpha$ sufficiently small we get a contradiction and the proof is complete concerning  a). The proof of the sharpness of b) is obtained by making an obvious modification of this argument so the proof is complete.
\end{proof}

  \vspace{4mm}
By doing similar calculations as in the proof of Theorem 2.4 in \cite{LEPNS12} (see also Theorem 7.10 in the book \cite{KPNS}), or just doing appropriate transformations, we obtain the following symmetric version of this equivalence Theorem:


 \begin{theorem}\label{Th2.3}
Let $ 0 < \ell \le \infty, $ let $p \in \mathbb{R}_+ \setminus
\{0\} $ and let $f$ be a non-negative function. Then\\
a) \ \ \ the inequality
\begin{equation}\label{2.4}
\int_0^\ell \left( \int_0^x
f(y)dy \right)^p x^{-\alpha} \frac{dx}{x}\le  \left(\frac{p}{\alpha}\right)^p
\int_0^\ell \left(xf(x)\right)^p x^{-\alpha} \left[
1-\left(\frac{x}{\ell}\right)^\frac{\alpha}{p} \right]\frac{dx}{x}
\end{equation}
holds
for all measurable functions $f, $ each $\ell, \ 0 < \ell \le \infty$ and all $\alpha$ in the following cases:
$$ \ \ \ \ \ \ \ \ \ \ \ \ (a_1) \ \ \ \ \ \ \ \ p\ge 1, \alpha>0,  $$
$$ \ \ \ \ \ \ \ \ \ \ \ \ (a_2) \ \ \ \ \ \ \ \ p< 0, \alpha<0.  $$\\
b) \ \ \ For the case $0<p<1, \ \alpha>0,$ inequality \eqref{2.4}
holds in the reversed direction for all $\ell, \ 0<\ell\leq \infty$.\\
c) \ \ \ \ \ The inequality
\begin{equation} \label{2.5}
 \int_\ell^\infty \left( \int_x^\infty f(y)dy
\right)^p x^{\alpha} \frac{dx}{x} \le  \left(\frac{p}{\alpha}\right)^p
\int_\ell^\infty f^p(x) x^{\alpha} \left[
1-\left(\frac{\ell}{x}\right)^\frac{\alpha}{p} \right]\frac{dx}{x}
\end{equation}
holds for all measurable functions $f, $ each $\ell, 0 \le \ell < \infty$ and all $\alpha$ in the following cases:
$$ (c_1) \ \ \ \ \ \ \ \ p\ge 1, \alpha>0,$$
$$ (c_2) \ \ \ \ \ \ \ \ p<0, a<0.$$\\
d) \ \ \ \ \ For the case $0 < p \le 1, \ \alpha >0$ inequality \eqref{2.5}
holds in the reversed direction for all $\ell, \ 0\leq\ell< \infty$.\\
e) \ \ \ \ \ All inequalities above are sharp.\\
f) \ \ \ \ \ Let $p\ge1$ or $p<0.$ Then, the statements in a) and c) are equivalent for all permitted $\alpha$.\\
g) \ \ \ \ \ Let $0<p<1.$ Then, the statements in b) and d) are equivalent for all permitted $\alpha.$
\end{theorem}

\begin{remark}\label{rem3.3} For the case $l = \infty$ the inequalities \eqref{2.4} and \eqref{2.5} were formulated, proved and applied in this convexity form in the new book \cite{PersTepWe}. This fact has further inspired us to reformulate our results in this convexity $(dx/x)$ way, which not only contribute to a better understanding but is also more suitably for such applications in modern harmonic analysis.
\end{remark}

\section{Reversed sharp Hardy inequalities for monotone functions}
\setcounter{equation}{0}\setcounter{theorem}{0}

For the proof of our main results in this Section we need the following Lemma:

\begin{lemma}\label{Lemma3.1}
Let $p > 0,$ $\frac{1}{p}+\frac{1}{q}=1$ and let $f$ be a non-negative and measurable function on $(a,b),$ $\infty \leq a < b\leq \infty.$

$a)$ Let $f$ be non-increasing on $(a, b),$ $\infty < a < b\leq \infty.$ If $p \geq 1,$ then
\begin{eqnarray}\label{3.1}
{{\left( \int\limits_{a}^{b}{f\left( y \right)dy} \right)}^{p}}\ge p\int\limits_{a}^{b}{{{\left( y-a \right)}^{p-1}}{{\left( f\left( y \right) \right)}^{p}}dy.}
\end{eqnarray}
If $0 < p \leq 1,$ then \eqref{3.1} holds in the reversed direction.

$b)$ Let $f$ be non-decreasing on $(a,b),$ $-\infty \le a<b<\infty .$ If $p \geq 1,$ then
\begin{eqnarray}\label{3.2}
{{\left( \int\limits_{a}^{b}{f\left( y \right)dy} \right)}^{p}}\ge p\int\limits_{a}^{b}{{{\left( b-y \right)}^{p-1}}{{\left( f\left( y \right) \right)}^{p}}dy.}
\end{eqnarray}
If $0 < p \leq 1,$ then \eqref{3.2} holds in the reversed direction.

$c)$ The constant $p$ is sharp in all these four inequalities. In fact, we have even equality in \eqref{3.1} for the function $f\left( y \right)=A{{\chi }_{\left( a,c \right)}}\left( y \right)$ for some $c\in \left( a,b \right)$ and $A > 0.$ Moreover, equality in \eqref{3.2} holds if $f\left( y \right)=A{{\chi }_{\left( c,b \right)}}\left( y \right)$ for some $c\in \left( a,b \right)$ and $A>0.$
\end{lemma}

Proofs of various variants of this Lemma can be found in many places (see e,g, \cite{BBP}) but for the readers convenience, we include a simple proof of just this variant.

\begin{proof}
First assume that $- \infty < a < b < \infty.$ Next we observe that the proof of $b)$ can be reduced to that of $a)$ by putting $g(y) =f(a+b-y).$ Hence, it is sufficient to prove $a).$ Moreover, by making suitable coordinate transformations we conclude that it is sufficient to consider the case $(a,b) = (0.1).$ Therefore, we consider a non-negative, measurable and non-increasing function $f$ on $(0,1).$

Let
$$F\left( x \right)=\int\limits_{0}^{x}{f\left( y \right)dy}.$$
Then $F(0) = 0,$ and, for almost all $x\in \left( 0,1 \right),$ if $p\ge 1$ then
$$\frac{d}{dx}{{\left( F\left( x \right) \right)}^{p}}=pf\left( x \right){{\left( F\left( x \right) \right)}^{p-1}}\ge p{{x}^{p-1}}{{\left( f\left( x \right) \right)}^{p}}$$

By integrating from $0$ to $1$ we find that
$${{\left( \int\limits_{0}^{1}{f\left( y \right)dy} \right)}^{p}}={{\left( F\left( 1 \right) \right)}^{p}}\ge p\int\limits_{0}^{1}{{{y}^{p-1}}f\left( y \right)dy.}$$

The same argument shows that this inequality holds in the reversed direction if $0 < p \leq 1.$ We conclude that $a)$ and $b)$ are proved. It is obvious that we have equality in the inequalities \eqref{3.1}  and  \eqref{3.2}  and their reversed versions for $0 < p \leq 1$  for the claimed test functions
$$f\left( y \right)=A{{\chi }_{\left( a,c \right)}}\left( y \right) \ \ \text{ and } \ \ \ f\left( y \right)=A{{\chi }_{\left( c,b \right)}}\left( y \right),$$
respectively.

The proof of the cases $a = -\infty$ or $b = \infty$ follows by just doing a limit procedure so the proof is complete.
\end{proof}

First we consider the case when $f$ is non-increasing and note that then such a reversed inequality has meaning only if $0 < \alpha < p$ (since if not the involved integrals diverges for all non-trivial functions $f$).

Our first main result reads:

\begin{theorem}\label{Theorem3.2}
Let $p >0,$ $0 < \alpha < p$ and let $f$ be a measurable, non-negative and non-increasing function on $(0,\ell),$ $0 < l \leq\infty.$

$a)$ If $p \geq1,$ then
\begin{eqnarray}\label{3.3}
 \int\limits_{0}^{\ell}{{{\left( \int\limits_{0}^{x}{f\left( y \right)dy} \right)}^{p}}{{x}^{-\alpha }}\frac{dx}{x}}
 \ge \frac{p}{\alpha }\int\limits_{0}^{x}{{{\left( xf\left( x \right) \right)}^{p}}{{x}^{-\alpha }}\left( 1-{{\left( \frac{x}{\ell} \right)}^{\alpha }} \right)\frac{dx}{x}.}
\end{eqnarray}

$b)$ If $0 < p \leq1,$ then \eqref{3.3} holds in the reversed direction.

$c)$ The constant $C = p/\alpha$ is sharp in both $a)$ and $b)$ and equality appears for each function $f\left( x \right)=A{{\chi }_{\left( 0,c \right)}}\left( x \right)$ for some $c\in \left( 0,l \right)$ and $A>0.$
\end{theorem}
\begin{proof}
$a)$ By using Lemma \ref{Lemma3.1} and Fubini´s theorem we find that
\begin{eqnarray*}
&&  \int\limits_{0}^{\ell}{{\left( \int\limits_{0}^{x}{f\left( y \right)dy} \right)}^{p}}{{x}^{-\alpha }}\frac{dx}{x}\\
&\ge& p\int\limits_{0}^{\ell}{\left( \int\limits_{0}^{x}{{{y}^{p-1}}{{\left( f\left( y \right) \right)}^{p}}dy} \right){{x}^{-\alpha }}}\frac{dx}{x} \\
& =&p\int\limits_{0}^{\ell}{{{\left( yf\left( y \right) \right)}^{p}}\left( \int\limits_{y}^{\ell}{{{x}^{-\alpha -1}}dx} \right)\frac{dy}{y}} \\
&=& \frac{p}{\alpha }\int\limits_{0}^{\ell}{{{\left( yf\left( y \right) \right)}^{p}}\left( {{y}^{-\alpha }}-{{\ell}^{-\alpha }} \right)\frac{dy}{y}} \\
&=& \frac{p}{\alpha }\int\limits_{0}^{\ell}{{{\left( yf\left( y \right) \right)}^{p}}{{y}^{-\alpha }}\left( 1-{{\left( \frac{y}{\ell} \right)}^{\alpha }} \right)\frac{dy}{y}.} \\
\end{eqnarray*}

$b)$ It is used only one inequality in the proof of $a)$ and, according to Lemma \ref{Lemma3.1}, this inequality holds in the reversed direction in this case so also $b)$ is proved.

$c)$ In view of the proofs above this sharpness statement follows by using Lemma \ref{Lemma3.1}  but we also verify this directly:
Let $f\left( x \right)=A{{\chi }_{\left( 0,c \right)}}\left( x \right),$ $c\in \left( 0,l \right).$ Then
\begin{eqnarray*}
 \frac{p}{\alpha }\int\limits_{0}^{\ell}{{{\left( xf\left( x \right) \right)}^{p}}{{x}^{-\alpha }}\left( 1-{{\left( \frac{x}{\ell} \right)}^{\alpha }} \right)\frac{dx}{x}}&=&\frac{p}{\alpha }{{A}^{p}}\int\limits_{0}^{c}{{{x}^{p-\alpha -1}}\left( 1-{{\left( \frac{x}{\ell} \right)}^{\alpha }} \right)dx} \\
&=& {{A}^{p}}\frac{p}{\alpha }\left( \frac{{{c}^{p-\alpha }}}{p-\alpha }-\frac{1}{{{\ell}^{\alpha }}}\frac{{{c}^{p}}}{p} \right):=I
\end{eqnarray*}
Moreover,
\begin{eqnarray*}
\int\limits_{0}^{\ell}{{\left( \int\limits_{0}^{x}{f\left( y \right)dy} \right)}^{p}}{{x}^{-\alpha }}\frac{dx}{x}&=&{{A}^{p}}\int\limits_{0}^{c}{{{x}^{p-\alpha -1}}dx+{{A}^{p}}{{c}^{p}}\int\limits_{c}^{\ell}{{{x}^{-\alpha -1}}dx}} \\
&=& {{A}^{p}}\frac{{{c}^{p-\alpha }}}{p-\alpha }+\frac{{{A}^{p}}{{c}^{p}}}{\alpha }\left( {{c}^{-\alpha }}-{{\ell}^{-\alpha }} \right)\\
&=&{{A}^{p}}\frac{p}{\alpha }{{c}^{p-\alpha }}-{{A}^{p}}\frac{{{c}^{p}}}{\alpha }=I \\
\end{eqnarray*}

We conclude that the constant $p/\alpha$ is sharp in both $a)$ and $b)$ with equality for
$$f\left( x \right)=A{{\chi }_{\left( 0,c \right)}}\left( x \right), \ \ c\in \left( 0,l \right),$$
so also  $c)$ is proved.

\end{proof}

As already mentioned the inequality \eqref{3.3} has no meaning in the cone of non-increasing functions if $\alpha \geq p.$ But it is not so if we instead restrict to the cone of non-decreasing functions. But in this case the “target function”
$$f\left( x \right)=\left( 1-{{\left( \frac{x}{\ell} \right)}^{\alpha }} \right)$$ is different and connected to the truncated ${{\beta }_{\alpha }}$ function defined as follows:
\[{{\beta }_{\alpha }}={{\beta }_{\alpha }}\left( u,v \right)=\int\limits_{\alpha }^{1}{{{t}^{u-1}}{{\left( 1-t \right)}^{v-1}},\,\,\,0\le \alpha <1.}\]
In particular ${{\beta }_{0}}$ coincides with usual $\beta$ function $\beta (u.v).$

Our next main result reads:

\begin{theorem}\label{Theorem3.3}
Let $\alpha \geq p> 0$ and let $f$ be a measurable, non-negative and non-decreasing function on $(0,\ell),$ $0 < \ell \leq \infty.$

$a)$ If $p \geq 1,$ then
\begin{eqnarray}\label{3.4}
&& \int\limits_{0}^{\ell}{{{\left( \int\limits_{0}^{x}{f\left( y \right)dy} \right)}^{p}}{{x}^{-\alpha }}\frac{dx}{x}\ge } \\  \notag
& \ge& \frac{p}{\alpha }\int\limits_{0}^{\ell}{{{\left( xf\left( x \right) \right)}^{p}}{{x}^{-\alpha }}T\left( x \right)\frac{dx}{x},}
\end{eqnarray}
where
$$T\left( x \right):=\alpha {{\beta }_{\frac{x}{\ell}}}\left( p,\alpha -p+1 \right),\,\,\,x\le l.$$

$b)$ If $0 < p \leq 1,$ then \eqref{3.4} holds in the reversed direction.

$c)$ the constant $\frac{p}{\alpha }$ is sharp in both $a)$ and $b)$ and equality appears if
$$f\left( x \right)=A{{\chi }_{\left( c,l \right)}}\left( x \right) \ \ \text{ for some } \ \ c\in \left( 0,l \right) \ \ \text{ and } \ \ A >0.$$
\end{theorem}
\begin{proof}
$a)$ By using again Lemma \ref{Lemma3.1} and Fubini´s theorem we obtain that
\begin{eqnarray*}
I&:=&\int\limits_{0}^{\ell}{{{\left( \int\limits_{0}^{x}{f\left( y \right)dy} \right)}^{p}}{{x}^{-\alpha }}\frac{dx}{x}\ge p\int\limits_{0}^{\ell}{\int\limits_{0}^{x}{{{\left( x-y \right)}^{p-1}}{{\left( f\left( y \right) \right)}^{p}}dy{{x}^{-\alpha }}}\frac{dx}{x}}} \\
&=& p\int\limits_{0}^{\ell}{{{\left( f\left( y \right) \right)}^{p}}\int\limits_{y}^{\ell}{{{\left( x-y \right)}^{p-1}}{{x}^{-\alpha }}}\frac{dx}{x}dy}.
\end{eqnarray*}
We make the transformation $t=\frac{y}{x}$ in the inner integral and get that
\begin{eqnarray*}
I&\ge& p\int\limits_{0}^{\ell}{{{\left( f\left( y \right) \right)}^{p}}\int\limits_{y/l}^{1}{{{\left( 1-t \right)}^{p-1}}{{\left( \frac{y}{t} \right)}^{p-\alpha -1}}\frac{dt}{t}dy}} \\
& =&p\int\limits_{0}^{\ell}{{{\left( yf\left( y \right) \right)}^{p}}{{y}^{-\alpha }}\int\limits_{y/l}^{1}{{{\left( 1-t \right)}^{p-1}}{{t}^{\alpha -p}}dt\frac{dy}{y}}} \\
&=& \frac{p}{\alpha }\int\limits_{0}^{\ell}{{{\left( yf\left( y \right) \right)}^{p}}{{y}^{-\alpha }}T\left( y \right)\frac{dy}{y}.}
\end{eqnarray*}

$b)$ Since the only inequality used above holds in the reversed direction in this case (see Lemma \ref{Lemma3.1}) the proof of $b)$ follows in the same way.

$c)$ Choose the test function
$$f\left( x \right)=A{{\chi }_{\left( c,l \right)}}\left( x \right), \ c\in \left( 0,l \right).$$
Then, in view of the proofs of $a)$ and $b),$ for any $p > 0$ the right hand side of \eqref{3.4} is equal to
\begin{eqnarray*}
	I&:=&\int\limits_{0}^{\ell}\int\limits_{0}^{x}\left(x-y\right)^{p-1}A^p\left({\chi }_{\left( c,l \right)}(y)\right)^pdyx^{-\alpha}\frac{dx}{x}\\
	&=& pA^p \int\limits_{c}^{\ell}\int\limits_{c}^{x}\left(x-y\right)^{p-1}dy x^{-\alpha}\frac{dx}{x}\\
	&=&A^p\int\limits_{c}^{\ell}\left(x-c\right)^p x^{-\alpha}\frac{dx}{x}
\end{eqnarray*}
Moreover, the left hand side of \eqref{3.4}  is equal to
$${{A}^{p}}\int\limits_{0}^{\ell}{\int\limits_{0}^{x}{{{\left( {{\chi }_{\left( c,l \right)}}\left( y \right)dy \right)}^{p}}{{x}^{-\alpha }}\frac{dx}{x}}={{A}^{p}}\int\limits_{c}^{\ell}{{{\left( x-c \right)}^{p}}{{x}^{-\alpha }}\frac{dx}{x}=I}}$$
so we have equality in \eqref{3.4}  and the reversed inequality for $0 < p \leq 1$ for all $p > 0.$

The proof is complete.
\end{proof}

\begin{example}\label{Example3.4}
For the case $l =\infty$ we obtain the sharp inequality
\begin{eqnarray*}
\int\limits_{0}^{\infty}\left(\int\limits_{0}^{x}f\left(y\right)dy\right)^px^{-\alpha}\frac{dx}{x}\geq p B(p, \alpha-p+1)\int\limits_{0}^{\infty}\left(xf(x)\right)^px^{-\alpha}\frac{dx}{x}
\end{eqnarray*}
for all non-decreasing functions $f$.  This inequality holds in the reversed direction when $0<p\le 1$ and the constant is sharp also then. Hence, by just changing notations we see that our result  generalizes also a result in \cite{BBP}.
\end{example}

Hence, we have investigated all cases concerning the usual (arithmetic mean) Hardy operator so we turn to the dual situation (c.f. Theorem \ref{Th2.3}
$c)$) and here the only non-trivial situation is to study the non-increasing case.

Our main result for this case reads:
\begin{theorem}\label{Theorem3.5} Let $p >0, \ \alpha>0$ and $f$ be a measurable, non-negative and non-increasing function on $(\ell,\infty),$ $0 \leq \ell < \infty.$
	
$a)$ If $p \geq 1,$ then
\begin{eqnarray}\label{3.5}
\int\limits_{\ell}^{\infty }{{{\left( \int\limits_{x}^{\infty }{f\left( y \right)dy} \right)}^{p}}{{x}^{\alpha }}\frac{dx}{x}\ge \frac{p}{\alpha }\int\limits_{\ell}^{\infty }{{{\left( xf\left( x \right) \right)}^{p}}{{x}^{\alpha }}{{T}_{0}}\left( x \right)\frac{dx}{x},}}
\end{eqnarray}
where
$${{T}_{0}}\left( x \right):=\alpha {{\beta }_{\frac{\ell}{x}}}\left( p,\alpha  \right),\,\,\,x\ge l.$$

$b)$ If $0 < p \leq 1,$ then \eqref{3.5} holds in the reversed direction.

$c)$ The constant $p/\alpha$ is sharp in both $a)$ and $b)$ and equality appears in both $a)$ and $b)$ if $$f\left( x \right)=A{{\chi }_{\left( \ell,c \right)}}\left( x \right) \ \ \text{ for some } \ \ c\in \left( \ell,\infty  \right) \ \ \text{ and } \ \ A >0.$$
\end{theorem}
\begin{proof}
$a)$ By again applying Lemma \ref{Lemma3.1} and Fubini´s theorem we get that
\begin{eqnarray*}
I&:=&\int\limits_{\ell}^{\infty }{{{\left( \int\limits_{x}^{\infty }{f\left( y \right)dy} \right)}^{p}}{{x}^{\alpha }}\frac{dx}{x}\ge p\int\limits_{\ell}^{\infty }{{{\left( y-x \right)}^{p-1}}{{\left( f\left( y \right) \right)}^{\beta }}dy{{x}^{\alpha }}\frac{dx}{x}}}\\
&=& p\int\limits_{\ell}^{\infty }{{{\left( f\left( y \right) \right)}^{p}}\int\limits_{\ell}^{y}{{{\left( y-x \right)}^{p-1}}{{x}^{\alpha -1}}dxdy}} \\
& =&p\int\limits_{\ell}^{\infty }{{{\left( f\left( y \right) \right)}^{p}}{{y}^{p-1}}\int\limits_{\ell}^{y}{{{\left( 1-\frac{x}{y} \right)}^{p-1}}{{x}^{\alpha -1}}dxdy}}
\end{eqnarray*}
Thus, by making the transformation $t = x/y$ in the inner integral we can conclude that
\begin{eqnarray*}
I&\ge& p\int\limits_{\ell}^{\infty }{{{\left( f\left( y \right)y \right)}^{p}}{{y}^{\alpha }}\int\limits_{l/y}^{1}{{{\left( 1-t \right)}^{p-1}}{{t}^{\alpha -1}}\frac{dy}{y}}} \\
& =&\frac{p}{\alpha }\int\limits_{\ell}^{\infty }{{{\left( f\left( y \right)y \right)}^{p}}{{y}^{\alpha }}{{T}_{0}}\left( x \right)\frac{dy}{y}.}
\end{eqnarray*}

$b)$ The proof follows in the same way since the only inequality used in $a)$ now holds in the reversed direction.

$c)$ Similarly as in the proof of Theorem \ref{Theorem3.3} $c)$ we can easily verify that we indeed has equality in the inequality \eqref{3.5} (and the reversed inequality when $0<p\leq 1$) for every function
$$f\left( x \right)=A{{\chi }_{\left( \ell,c \right)}}\left( x \right), \ c\in \left( \ell,\infty \right) \ \text{and} \  A> 0.$$
Hence, also the sharpness is proved.
\end{proof}

\begin{example}\label{Example3.6}
Let $f,$ $p$ and $\alpha$ be defined as in Theorem \ref{3.5}. If $p \geq 1,$ then
$$\int\limits_{0}^{\infty }{{{\left( \int\limits_{x}^{\infty }{f\left( y \right)dy} \right)}^{p}}{{x}^{\alpha }}\frac{dx}{x}\ge p\beta \left( p,\alpha  \right)\int\limits_{0}^{\infty }{{{\left( xf\left( x \right) \right)}^{p}}{{x}^{\alpha }}\frac{dx}{x},}}$$
where $f(x)$ is a non-negative and non-increasing function.
The inequality holds in the reversed direction when $0<p\le 1$ and the constant $p\beta \left( p,\alpha  \right)$ is sharp in both cases. Hence, Theorem \ref{Theorem3.5} may be regarded also as generalization of another result in \cite{BBP}.
\end{example}

  \section{Applications}
   \setcounter{equation}{0}\setcounter{theorem}{0}

\qquad By combining Theorem \ref{Th2.3} $a),$ $b)$ and $c)$ with Theorem \ref{Theorem3.2} we obtain the following sharp two sided estimates:

\begin{theorem}\label{Theorem4.1}. Let $p >0,$ $o<\alpha <p, \ 0 < \ell \leq \infty$ and let $f$ be a measurable, non-negative and non-increasing function on $(0, \ell).$

If $p > 1,$ then
\begin{eqnarray}\label{4.1}
{{\left( \frac{p}{\alpha } \right)}^{1/p}}{{I}_{1}}\le {{I}_{2}}\le \frac{p}{\alpha }{{I}_{1}},
\end{eqnarray}
where
$${{I}_{1}}={{\left( \int\limits_{0}^{\ell}{{{\left( xf\left( x \right) \right)}^{p}}{{x}^{-\alpha }}\left( 1-{{\left( \frac{x}{\ell} \right)}^{\alpha }} \right)\frac{dx}{x}} \right)}^{1/p}}$$
and
$${{I}_{2}}=\int\limits_{0}^{\ell}{{{\left( \int\limits_{0}^{x}{f\left( y \right)dy} \right)}^{p}}{{x}^{-\alpha }}\frac{dx}{x}.}$$
If $0 < p \leq 1,$ then \eqref{4.1} holds in the reversed direction. Moreover, both constants
${{(p/\alpha )}^{1/p}}$ and $p/\alpha$  are sharp for all $p > 0.$
\end{theorem}

\begin{remark}\label{Remark4.2}
$a)$ This means that the equivalence ${{I}_{2}}\approx {{I}_{1}}$ holds and the corresponding “optimal target function “ is
$$g\left( x \right)=1-{{\left( \frac{x}{\ell} \right)}^{\alpha }}.$$

$b)$ In the lower inequality we can even have equality while in the above inequality the sharpness follows by choosing a sequence of non-increasing functions (a well-known fact from the theory of Hardy-type inequalities).
\end{remark}

\begin{remark}\label{Remarks4.3}
Many crucial objects in different mathematical areas are non-decreasing (e.g. in Lorentz spaces, interpolation theory, approximation theory and harmonic analysis). Hence, in particular, Theorem \ref{Theorem4.1} can be useful to obtain some more precise versions of known results in each of these areas. We illustrate this fact only in the theory of Lorentz spaces but aim to later also use our result to improve some results in the modern harmonic analysis as presented in the new book \cite{ PersTepWe}.
\end{remark}

Let ${{f}^{*}}$ denote the non-increasing rearrangement of a function f on a measure space $\left( \Omega ,\mu  \right).$ The Lorentz spaces $L^{p,q},$  $0 < p,q < \infty$ are defined by using the quasi.norm (norm when $p > 1,$ $q \geq 1$)
\begin{eqnarray}\label{4.2}
\left\| f \right\|_{p,q}^{*}:={{\left( \int\limits_{0}^{\infty }{{{\left( {{f}^{*}}\left( t \right){{t}^{1/p}} \right)}^{q}}\frac{dt}{t}} \right)}^{1/q}}.
\end{eqnarray}
It is well-known that for the case $p > 1$ this quasi-norm is equivalent to the following one equipped with the usual Hardy operator:
\begin{eqnarray*}
\left\| f \right\|_{p,q}^{**}:={{\left( \int\limits_{0}^{\infty }{{{\left( \int\limits_{0}^{t}{{{f}^{*}}\left( u \right)du} \right)}^{q}}{{t}^{-q/{p}'}}\frac{dt}{t}} \right)}^{1/q}}.
\end{eqnarray*}
Moreover, we have the following more precise estimates:
\begin{eqnarray}\label{4.3}
{{\left( {{p}'} \right)}^{1/q}}\left\| f \right\|_{p,q}^{*} \le \left\| f \right\|_{p,q}^{**}\le {p}'\left\| f \right\|_{p,q}^{*}
\end{eqnarray}
if $q > 1$ and the reversed inequalities hold if $0 < q \leq 1.$
However, by using Theorem \ref{Theorem4.1} we not only get the sharp estimates in \eqref{4.3} but also the following more precise statement:

\begin{corollary}\label{Corollary4.4}
With the notations and assumptions above, $p > 1$ and $0 < \ell \leq \infty,$ we have that
\begin{eqnarray}\label{4.4}
{{\left( {{p}'} \right)}^{1/q}}I_{\ell}^{*}\le I_{\ell}^{**}\le {p}'I_{\ell}^{*},
\end{eqnarray}
where $q > 1,$
$$I_{\ell}^{*}:={{\left( \int\limits_{0}^{\ell}{{{\left( {{f}^{*}}\left( t \right){{t}^{1/p}} \right)}^{q}}\left( 1-{{\left( \frac{t}{\ell} \right)}^{q/{p}'}} \right)\frac{dt}{t}} \right)}^{1/q}}$$
and
$$I_{\ell}^{**}:={{\left( \int\limits_{0}^{\ell}{{{\left( \int\limits_{0}^{t}{{{f}^{*}}\left( u \right)du} \right)}^{q}}{{t}^{-q/{p}'}}\frac{dt}{t}} \right)}^{1/q}}.$$
If $0 < q \leq 1,$ then the inequalities in \eqref{4.4} hold in the reversed directions. Both constants $({p}')^{1/q}$ and ${p}'$ are sharp for all $q >0.$
\end{corollary}
\begin{proof}
Just apply Theorem \ref{Theorem4.1} with $p$ replaced by $q$ and $\alpha$ replaced by $q/{p}'.$
\end{proof}

\begin{remark}\label{Remark4.5}
Note that \eqref{4.3} is obtained by just using \eqref{4.4} with $l = \infty,$ so in particular, both constant in \eqref{4.3} (and the reversed inequalities for $0\leq q \leq 1$) are sharp.
\end{remark}

\begin{remark}\label{Remark4.6}
For the case $0 < p \leq 1$ it is known that the quasi-norm $\left\| f \right\|_{p,q}^{*}$ is equivalent to the following quasi-norm $\left\| f \right\|_{p,q}^{**}$ equipped with the dual Hardy operator:
$$\left\| f \right\|_{p,q}^{**}:={{\left( \int\limits_{0}^{\infty }{{{\left( \int\limits_{t}^{\infty }{{{f}^{*}}\left( u \right)du} \right)}^{q}}{{t}^{-q/{p}'}}\frac{dt}{t}} \right)}^{1/q}}$$

By instead using Theorem \ref{Th2.3} $c),$ $d)$ and $e)$ with $l = 0$ combined with Example \ref{Example3.6} we obtain that if $0 < p \leq 1,$ then
\begin{eqnarray}\label{4.5}
{{\left( q\beta \left( q,-q/{p}' \right) \right)}^{1/q}}\left\| f \right\|_{p,q}^{*}\le \left\| f \right\|_{p,q}^{**}\le -{p}'\left\| f \right\|_{p,q}^{*}
\end{eqnarray}
if $q \geq 1$ and the reversed inequalities hold if $0 < q \leq 1.$ Both constants ${{\left( q\beta \left( q,-q/{p}' \right) \right)}^{1/q}}$  and $-{p}'$ are sharp for all $q > 0.$
\end{remark}

\begin{remark}\label{Remark4.7}
A more general statement like that  in Corollary \ref{Corollary4.4} involving sharp constants in both inequalities can be formulated, where the integrals $\int\limits_{0}^{\infty }{{}}$ are replaced by the integrals $\int\limits_{\ell}^{\infty }{{}},$ $0  \leq \ell < \infty.$ In particular, this gives a similar generalization of \eqref{4.5}. However, in this case the result looks less nice since the two target functions
$1-{{\left( \frac{x}{\ell} \right)}^{\alpha }}$ and
$\alpha {{\beta }_{\frac{\ell}{x}}}\left( p,\alpha  \right)$ do not coincide.
\end{remark}

\vspace{2mm}

We only give the following final example related to Remark \ref{Remark4.7} and the well-known  inequality: If $0 < p < 1,$ then

\begin{eqnarray}\label{4.6}
\int\limits_{0}^{\infty }{{{\left( \frac{1}{x}\int\limits_{x}^{\infty }{f\left( y \right)dy} \right)}^{p}}dy\le \frac{\pi p}{\sin \pi p}\int\limits_{0}^{\infty }{{{f}^{p}}\left( x \right)dx}}
\end{eqnarray}
for all functions as defined in Theorem \ref{Theorem3.5}:

\begin{example}\label{Remark4.8}
Let $0 < p < 1$ and let $f$ be a measurable, non-negative and non-increasing function on $\left( \ell,\infty  \right),\,\,\,0\le l<\infty .$ If $0<p\leq 1,$ then

\begin{eqnarray*}
\int\limits_{\ell}^{\infty }{{{\left( \frac{1}{x}\int\limits_{0}^{\infty }{f\left( y \right)dy} \right)}^{p}}dx\le p\int\limits_{\ell}^{\infty }{{{f}^{p}}\left( x \right){{\beta }_{\frac{\ell}{x}}}\left( p,1-p \right)dx}},
\end{eqnarray*}
and the constant $p$ is sharp. This is just Theorem \ref{Theorem3.5} b) with $\alpha=1-p.$

In particular, for $l=\infty$ this inequality coincides with \eqref{4.6} since

$$\beta \left( p,1-p \right)=\pi /\sin \pi p,$$
so the constant $\frac{\pi p}{\sin \pi p}$ in \eqref{4.6} is sharp.
\end{example}

 \section{Some further results and final remarks}
  \setcounter{equation}{0}\setcounter{theorem}{0}

\quad \ First we remark that e.g the Hardy inequality \eqref{2.4} has no meaning in the limit case $\alpha= 0.$ However, by restricting to the interval $(0,1)$ and involving some suitable logarithms C. Bennett in 1973 succeeded to prove such an inequality when he developed his well-known theory for real interpolation between the (fairly close) spaces $L$ and $LLog L$ on $(0,1),$ see \cite{Bennett} and c.f. also \cite{BBP}. This result has been generalized by other authors but the so far most precise results were derived in \cite{SBLEPNS14}. Here we state a little more general form of this result in our $dx/x$ terminology and with the interval $(0,1)$ replaced by $(0,\ell), \ 0<\ell<\infty.$

\begin{theorem}\label{t5.1}
Let $\alpha, p>0$ and $f$ be a non-negative
and measurable function on $(0, \ell), \ 0<\ell<\infty.$\\

$(a)$ \ \ If $p>1,$ then
\begin{eqnarray}\label{5.1}
&&\alpha^{p-1} \left(  \int_0^\ell f(x) dx \right)^p + \alpha^p
 \int_0^\ell [ \log (\ell/x)]^{\alpha p - 1} \left( \int_0^x
f(y) dy \right)^p \frac{ dx}{x} \\ \notag
 &\le &  \int_0^\ell x ^p [ \log (\ell/x)]^{(1 + \alpha) p - 1} f^p (x) \frac{ dx}{x}.
\end{eqnarray}
Both constants $\alpha^{p-1}$ and $\alpha^p$ in
\eqref{5.1} are sharp. \\

$(b)$ \ \ If $0 < p < 1, $ then \eqref{5.1} holds in the reverse direction and both constants $\alpha^{p-1}$ and $\alpha^p$ are sharp. \\

$(c)$ \ \ If $p=1,$ then we have equality in \eqref{5.1}.
\end{theorem}
\begin{proof}
The proof can be done by just modifying step by step the arguments in the proof for the case $l = 1.$ Alternatively the result can be derived by using the original result in \cite{SBLEPNS14} and making suitable variable substitutions. Hence, we omit the details.
\end{proof}

\begin{remark}\label{Remark_5.2}
\eqref{5.1} is one of the few inequalities we know containing two constant and both are sharp. In the original paper \cite{Bennett} only the case (a) was considered and with one constant involved (the first term in \eqref{5.1} was missed) and the sharpness was not discussed at all.
\end{remark}

\begin{remark}\label{Remark_5.3}
By using Theorem \ref{t5.1} with $f(x) = g(1/x)x^{-2}$ and making obvious variable transformations and changes in notations we also get the following “dual” version:
\end{remark}

\begin{theorem}\label{t5.4}
	Let $\alpha, p>0$ and $f$ be a non-negative
	and measurable function on $(0, \ell), \ 0<\ell<\infty$\\
	$(a)$ \ \ If $p>1,$ then
	\begin{eqnarray}\label{5.2}
	&&\alpha^{p-1} \left(  \int_\ell^\infty f(x) dx \right)^p + \alpha^p
	\int_\ell^\infty [ \log (xe/\ell)]^{\alpha p - 1} \left( \int_x^\infty
	f(y) dy \right)^p \frac{ dx}{x} \\ \notag
	&\le &  \int_\ell^\infty x ^p [ \log (xe/\ell)]^{(1 + \alpha) p - 1} f^p (x) \frac{ dx}{x}.
	\end{eqnarray}
	Both constants $\alpha^{p-1}$ and $\alpha^p$ in
	\eqref{5.2} are sharp. \\
	$(b)$ \ \ If $0 < p \leq 1, $ then \eqref{5.2} holds in the reverse direction and also here both constants $\alpha^{p-1}$ and $\alpha^{p}$ are sharp.
\end{theorem}

Next we pronounce that all sharp inequalities we presented so far are for the case $q = p.$ Very little concerning sharp constants is known for other cases. Let us illustrate this problem by mentioning the fact that by applying the general theory in Hardy-type inequalities (see e.g. the book \cite{KPNS}) in a power weighted case we get in our $dx/x$ frame the following:

\begin{example}\label{Ex5.5}
	The inequality
	\begin{eqnarray}\label{5.3}
	\left(\intl_0^\infty \left(
	\intl_0^xf(t)dt\right)^q x^{-\al} \frac{dx}{x} \right)
	^\frac{1}{q} \le C
	\left(\intl_0^\infty ( xf(x))^p x^{-\bt} \frac{dx}{x}\right)^\frac{1}{p}
	\end{eqnarray}
	holds for some finite constant $C>0$ for $1<p\le q<\infty$ , if and only if
	\begin{eqnarray}\label{5.4}
	\bt>0 \ \ \ {\rm and } \ \ \frac{\alpha}{q}
	=\frac{\bt}{p}.
	\end{eqnarray}
\end{example}

\begin{remark}\label{Remark_5.6}
For the case $p = q$ we have already pointed out the sharp constant but for the case $1 < p < q < \infty$ this has been a fairy long lasted open question since G.A. Bliss in 1930 solved it for $\beta = p-1$ (see \cite{BL}). It was finally solved in 2015 in the paper \cite{LEPSS15} and in our $dx/x$ frame their result reads:
\end{remark}

\begin{theorem}\label{t5.7}
	Let $1<p<q<\infty$ and the parameters  $\al$  and $\beta$ satisfying \eqref{5.4}. Then the sharp constant in \eqref{5.3}  is $C=C_{pq}^\ast,$ where
	\begin{equation}\label{5.5a}
	C_{pq}^\ast= \left(\frac{p-1}{\bt}\right)^{\frac{1}{p^\prime}+\frac 1q}
	\left(\frac{p^\prime}{q}\right)^\frac{1}{p}
	\left(\frac{\frac{q-p}{p}\Gm\left(\frac{pq}{q-p}\right)}
	{\Gm\left(\frac{p}{q-p}\right)
		\Gm\left(\frac{p(q-1)}{q-p}\right)}\right)
	^{\frac{1}{p}-\frac 1q}.
	\end{equation}
\end{theorem}
\begin{remark}\label{Remark_5.8} 	
Some straightforward calculations show that
	$$
	C_{pq}^\ast \to \frac{p}{\bt} \ {\rm as} \ q\to p,
	$$
so indeed we have the expected continuity in the sharp constants as $\ q\to p.$
\end{remark}

In the dual situation we have the following:

\begin{example}\label{Ex5.9}
	The inequality
	\begin{eqnarray}\label{5.5}
	\left(\intl_0^\infty \left(
	\intl_x^\infty f(t)dt\right)^q x^{\al} \frac{dx}{x} \right)
	^\frac{1}{q} \le C
	\left(\intl_0^\infty ( xf(x))^p x^{\bt} \frac{dx}{x}\right)^\frac{1}{p}
	\end{eqnarray}
	holds for $1<p\le q<\infty$ and some finite constant $C>0,$ if and only if
	\begin{eqnarray*}
	\bt>0 \ \ \ {\rm and } \ \ \frac{\alpha}{q}
	=\frac{\bt}{p}
	\end{eqnarray*}
and the sharp constant is known also in this case (see\cite{LEPSS15}).
\end{example}

\begin{remark}\label{Remark_5.10}
Also all cases when we have equality in \eqref{5.3} with $C=C^\star_{pq}$ defined by \eqref{5.5a} and when we have equality in \eqref{5.5} are also known (see again \cite{LEPSS15}). Hence, it seems to be an interesting open question to derive the corresponding sharp results when the integrals $\int_0^\infty$ are replaced by $\int_0^\ell, \ 0<\ell\leq \infty$ or $\int_\ell^\infty, \ 0 \le \ell < \infty,$ respectively. We aim to investigate this in a forthcoming paper. We use this opportunity to note a misprint in [14]. The condition $\frac{n+\alpha}{p} = \frac{n+\beta}{q}$ in Theorems 4.1 and 4.2 in [14], should be replaced by $\frac{n+\alpha}{q} = \frac{n+\beta}{p}.$
\end{remark}

\begin{remark}\label{Remark_5.11}
By using the same transformations as those pointed out in Remark \ref{Remark_5.3}  we can transform inequalities involving integrals $\int_0^\ell$ to inequalities involving the integrals $\int_\ell^\infty.$ Let us just as one example of this fact restate Theorem \ref{Theorem3.2}  in this way:
\end{remark}

\begin{theorem}\label{t5.12}
Let $p > 0, 0 < \alpha < p$ and let $f(x)x^2$ be a measurable, non-negative and non-decreasing function on $(0,\ell), 0 \leq \ell < \infty.$

a)	It $p \geq 1,$ then
	\begin{eqnarray}\label{5.6}
\left(\intl_\ell^\infty \left(
\intl_x^\infty f(y)dy\right)^p x^{\al} \frac{dx}{x} \right)
^\frac{1}{p} \geq \frac{p}{\alpha}
\left(\int_\ell^\infty ( xf(x))^p x^\alpha\left(1-\left(\frac{\ell}{x}\right)^\alpha\right) \frac{dx}{x}\right)^\frac{1}{p}
\end{eqnarray}

b)	If $0 < p\leq 1,$ then \eqref{5.6} holds in the reversed direction.

c)	The constant $p/\alpha$ is sharp in both a) and b) and equality appears for any $f (x) = Ax^{-2}\chi_{(c,\infty)}(x)$ for some $c\in(\ell, \infty), \ A>0$.
\end{theorem}

\begin{remark}\label{Remark_5.13}
The function f(x) in Theorem \ref{t5.12} is an example of a so called quasi-monotone function, which means that  $f(x)x^\alpha$ is non-increasing or non-decreasing for some $\alpha\in R.$ It is another interesting open question to investigate all our results concerning monotone functions for such more general quasi-monotone functions.  Even in the case with infinite intervals some interesting phenomena appear. See \cite{BBP} and the references therein for a special case.
\end{remark}

\bibliographystyle{plain}
\bibliography{LEPNSGT22}

\end{document}